\pdfoutput=1
\RequirePackage{ifpdf}
\ifpdf 
\documentclass[pdftex]{sigma}
\else
\documentclass{sigma}
\fi

\usepackage{mathrsfs}
\input xy
\xyoption{all}

\renewcommand{\AA}{\mathbb{A}}
\newcommand{\CC}{\mathbb{C}}
\newcommand{\HH}{\mathbb{H}}
\newcommand{\PP}{\mathbb{P}}
\newcommand{\QQ}{\mathbb{Q}}
\newcommand{\ZZ}{\mathbb{Z}}
\newcommand{\sD}{\mathscr{D}}
\newcommand{\sG}{\mathscr{G}}
\newcommand{\sK}{\mathscr{K}}
\newcommand{\sM}{\mathscr{M}}
\newcommand{\sO}{\mathscr{O}}
\newcommand{\de}{\mathrm{d}}
\newcommand{\dR}{\mathrm{dR}}
\newcommand{\Hig}{\mathrm{Hig}}
\newcommand{\coH}{\mathrm{H}}
\newcommand{\dRH}{\diamond}
\newcommand{\ddRH}{\Theta}
\newcommand{\cpt}{compactification}

\newcommand{\flr}[1]{\lfloor #1 \rfloor}
\newcommand{\Gr}{\operatorname{Gr}}
\newcommand{\ord}{\operatorname{ord}}
\newcommand{\gen}[1]{\left\langle #1 \right\rangle}
\newcommand{\RD}{\mathrm{R}}
\newcommand{\red}[1]{{#1}_{\mathrm{red}}}
\newcommand{\HN}{\mathrm{HN}}
\newcommand{\rest}[1]{_{|{#1}}}

\begin{document}
\allowdisplaybreaks

\newcommand{\arXivNumber}{1608.08700}

\renewcommand{\thefootnote}{}

\renewcommand{\PaperNumber}{055}

\FirstPageHeading

\ShortArticleName{The K\"unneth Formula for the Twisted de Rham and Higgs Cohomologies}

\ArticleName{The K\"unneth Formula for the Twisted de Rham\\ and Higgs Cohomologies\footnote{This paper is a~contribution to the Special Issue on Modular Forms and String Theory in honor of Noriko Yui. The full collection is available at \href{http://www.emis.de/journals/SIGMA/modular-forms.html}{http://www.emis.de/journals/SIGMA/modular-forms.html}}}

\Author{Kai-Chieh CHEN~$^\dag$ and Jeng-Daw YU~$^\ddag$}

\AuthorNameForHeading{K.-C.~Chen and J.-D.~Yu}

\Address{$^\dag$~Department of Mathematics, University of California, Berkeley, Berkeley, CA, USA}
\EmailD{\href{mailto:kaichiehchen@berkeley.edu}{kaichiehchen@berkeley.edu}}
\URLaddressD{\url{https://math.berkeley.edu/~kaichieh/}}

\Address{$^\ddag$~Department of Mathematics, National Taiwan University, Taipei, Taiwan}
\EmailD{\href{mailto:jdyu@ntu.edu.tw}{jdyu@ntu.edu.tw}}
\URLaddressD{\url{http://homepage.ntu.edu.tw/~jdyu/}}

\ArticleDates{Received February 23, 2018, in final form June 02, 2018; Published online June 12, 2018}

\Abstract{We prove the K\"unneth formula for the irregular Hodge filtrations on the exponentially twisted de Rham and the Higgs cohomologies of smooth quasi-projective complex varieties. The method involves a careful comparison of the underlying chain complexes under a certain elimination of indeterminacy.}


\Classification{14F40; 18F20; 14C30}

\renewcommand{\thefootnote}{\arabic{footnote}}
\setcounter{footnote}{0}

\section{The main result}

Let $U$ be a smooth quasi-projective variety over the field $\CC$ of complex numbers, and $f\in\Gamma(U,\sO_U)$ a regular function. Attached to such a pair $(U,f)$ and a non-negative integer $k$, one has the
\textit{$k$-th de Rham cohomology $\coH_\dR^k(U,f)$} and \textit{Higgs cohomology $\coH_\Hig^k(U,f)$}, defined in Section~\ref{sect:Fil}. The two spaces $\coH_\dRH^k(U,f), \dRH\in\{\dR, \Hig\}$, are of the same finite dimension over $\CC$ (see \cite[Remark~1.3.3]{ESY}), and are equipped with the decreasing \textit{irregular Hodge filtrations}
\begin{gather*} F^\lambda \coH_\dRH^k(U,f), \qquad \lambda \in \QQ, \end{gather*}
indexed by the ordered set $\QQ$ of rational numbers with finitely many jumps. In the following, we omit the adjective and just call them the \textit{Hodge filtrations}. For the motivations and the basic properties of the Hodge filtration, including in particular the degeneration of the Hodge to de Rham spectral sequence, see \cite{ESY, KKP} and the references therein. We recall the construction in Section~\ref{sect:Fil}.

Now consider two such pairs $(U_i,f_i)$, $i=1,2$. On the product $U:=U_1\times U_2$, consider the regular function $f$ defined by $f\colon (x_1,x_2) \mapsto f_1(x_1) + f_2(x_2)$. We call the pair $(U,f)$ the \textit{product} of the two $(U_i,f_i)$. For $\dRH \in \{\dR, \Hig\}$, there is the canonical map
\begin{gather}\label{eq:H-prod}
\bigoplus_{i+j=k} \coH_\dRH^i(U_1,f_1) \otimes \coH_\dRH^j(U_2,f_2) \to \coH^k_\dRH(U,f)
\end{gather}
induced by cup product. We equip the space on the left hand side with the product filtration, i.e., the $\lambda$-th filtration for $\lambda\in\QQ$ is given by the subspace
\begin{gather*} \bigoplus_{i+j=k}\left(\sum_{a +b = \lambda} F^a\coH^i_\dRH(U_1,f_1)\otimes F^b\coH^j_\dRH(U_2,f_2) \right), \end{gather*}
where the inner sum is taken inside $\coH_\dRH^i(U_1,f_1)\otimes \coH_\dRH^j(U_2,f_2)$. In this article, we prove the following K\"unneth formula.

\begin{theorem}\label{Thm:Kunneth}
With notations as above, the map \eqref{eq:H-prod} is an isomorphism of filtered spaces.
\end{theorem}

In particular, denoting $\Gr_F^\lambda V$ the $\lambda$-th graded piece of a filtered space $(V,F)$, one has the identification
\begin{gather*} \Gr_F^\lambda\coH_\dRH^k(U,f) = \bigoplus_{r,s} \Gr_F^r\coH_\dRH^s(U_1,f_1) \otimes \Gr_F^{\lambda-r}\coH_\dRH^{k-s}(U_2,f_2). \end{gather*}

Recently different proofs of the K\"unneth formula in a more general setting appear in \cite[Theorem~3.39]{S2} where the involved coefficients in the cohomology are allowed to be exponential twists of complex Hodge modules. One of the main ideas in that work is to enrich the de Rham and Higgs cohomologies into the Brieskorn lattice (see the last section Section~\ref{sect:BL}) or a~twistor structure and treat the irregular Hodge filtration as a byproduct of the enrichment. The notion of a $V$-adapted trivializing lattice for a meromorphic connection on $\PP^1$ of special type is introduced \cite[Section~3.2.b]{S2} in order to obtain the K\"unneth formula for irregular Hodge filtrations. The methods depend on the deep theory of twistor $\sD$-modules mainly developed by Sabbah and Mochizuki (see~\cite{M_twistor}). On the other hand, our approach is much more elementary. We believe that the concrete filtered de Rham complex used here would be suitable for computations in some interesting examples of irregular Hodge structures in the future work.

The rest of the article begins in Section~\ref{sect:Fil} with a brief of the construction of the Hodge filtration. We follow the approach of~\cite{Y} by putting a filtration on the de Rham complex or the Higgs complex via a certain {\cpt} of the pair $(U,f)$. Here we introduce the notion of a non-degenerate {\cpt}, which is weaker than that of a good {\cpt} used in~\cite{Y}, but appears naturally in the later section (see also \cite[Section~4]{S}, \cite[Section~7.3]{SY}, \cite{M_GKZ}). In order to compare the cohomologies with the filtrations of the summands $(U_i,f_i)$ and of their product $(U,f)$, we construct a particular {\cpt} of $(U,f)$
from the fixed ones of $(U_i,f_i)$ in Section~\ref{sect:Resolution}. The proof of the K\"unneth formula is obtained by a careful investigation of the relations between the involved complexes on the {\cpt}s. In the last Section~\ref{sect:BL}, we remark that one can interpolate the two spaces $\coH_\dR^k(U,f)$ and $\coH_\Hig^k(U,f)$ as the fibers of the \textit{Kontsevich bundle} on the projective line $\PP^1$ over $1$ and $0$, respectively.
In fact, the fiber over $c\in \CC\setminus \{0\}$ of the bundle is equal to $\coH_\dR^k(U,f/c)$ and hence the K\"unneth formula also holds true. However at $\infty$ the situation is more complicated in this regard and the direct analogue of the K\"unneth formula does not hold in general.

\section{The filtrations}\label{sect:Fil}

In this section, we fix a pair $(U,f)$ consisting of a smooth quasi-projective variety $U$ over $\CC$ and a regular function $f\in\Gamma(U,\sO_U)$.

\subsection{The \cpt}
Let $X$ be a projective variety over $\CC$ containing $U$ such that the reduced subvariety $S := X\setminus U$ is a normal crossing divisor. Regard $f\in\Gamma(X,\sO_X(*S))$ as a rational function on $X$. Let~$P$ and~$Z$ be the pole divisor and the zero divisor of $f$ on $X$, respectively, and let $\red{P}$ be the support of~$P$.

\begin{definition*}\quad\begin{enumerate}\itemsep=0pt
\item[(i)] We call $X$ a \textit{non-degenerate {\cpt}} of $(U,f)$ if there exists a neighborhood $V \subset X$ of $P$ such that $Z\cap V$ is smooth and $Z+S$ forms a reduced normal crossing divisor on $V$.

\item[(ii)] We call $X$ a \textit{good {\cpt}} of $(U,f)$ if $f$ indeed defines a morphism $f\colon X \to \PP^1$.
\end{enumerate}
\end{definition*}

If $X$ is non-degenerate, analytically locally at a point of $P$, there exists a coordinate system
\begin{gather*} \{x_1,\dots,x_l, y_1,\dots,y_m,z_1,\dots,z_r\} \end{gather*}
such that
\begin{gather}\label{eq:non-deg-local}
S = (xy) \qquad\text{and}\qquad f= \frac{1}{x^e} \qquad\text{or}\qquad f = \frac{z_1}{x^e}
\end{gather}
for some $e \in \ZZ_{>0}^l$. (Here and afterwards, we use the standard multi-index convention.) If $X$ is good, the second case of $f$ in the above expression does not occur.

For example, consider the case where $U$ is a complex torus and $f$ a Laurent polynomial. Assume that $f$ is non-degenerate with respect to its Newton polytope (for the definition, see \cite[Section~4]{Y}). Then any toric smooth {\cpt} compatible with the Newton polytope is a non-degenerate {\cpt} of~$(U,f)$ by \cite[Proposition~4.3]{Y} or \cite[Lemma~6.6]{M_GKZ}. The non-degenerate {\cpt} also appears in the considerations of rescaling
from a good {\cpt} \cite[Section~7.3]{SY}, and of Fourier transform \cite[Section~4]{S}. It is discussed in~\cite{M_GKZ} where in this situation, the author calls the meromorphic function $f$ on $X$ \textit{non-degenerate along~$S$} \cite[Definition~2.6]{M_GKZ}. In this case, the author investigates the structures of the twistor $\sD$-module associated with the meromorphic connection $(\sO_X(*S), \de +\de f)$; it is shown in particular that if $S$ equals the pole divisor of~$f$, the resulting twistor $\sD$-module is pure \cite[Lemma~2.10 and Corollary~3.12]{M_GKZ}.

\subsection{The filtered complexes}
Fix a non-degenerate {\cpt} $X$ of $(U,f)$ with boundary $S$. Regard $f$ as a rational function on $X$ and let $P = f^*(\infty)$ be the pole divisor with multiplicities. We have $\de f \in \Gamma\big(X,\Omega_X^1(\log S)(P)\big)$.

\begin{definition*}\quad\begin{enumerate}\itemsep=0pt
\item[(i)] The \textit{twisted de Rham complex} and the \textit{Higgs complex} are the complexes
\begin{gather*}
(\Omega_X^\bullet(\log S)(*P), \ddRH) = \big[ \sO_X(*P) \!\xrightarrow{\ddRH} \!\cdots\! \xrightarrow{\ddRH} \! \Omega_X^i(\log S)(*P) \!\xrightarrow{\ddRH}\! \Omega_X^{i+1}(\log S)(*P)\xrightarrow{\ddRH}\!\cdots \big],
\end{gather*}
where $\ddRH = \de + \de f$ and $\de f$, respectively. Here $\de f$ sends a local section $\omega$ to $\de f\wedge \omega$.

\item[(ii)] We call the associated hypercohomology groups $\HH^k(X,(\Omega_X^\bullet(\log S)(*P), \ddRH))$ the \textit{de Rham cohomology} and the \textit{Higgs cohomology of $(U,f)$}, and denote respectively by
\begin{gather*} \coH_\dR^k(U,f) = \coH^k(U,\de +\de f) \qquad\text{and}\qquad \coH_\Hig^k(U,f) = \coH^k(U,\de f). \end{gather*}

\item[(iii)] For an effective divisor $D$ on $X$ and $\mu\in\QQ$, let
\begin{gather*} \Omega_X^i(\log S)(\flr{\mu D})_+ = \begin{cases}
\Omega_X^i(\log S)(\flr{\mu D}), & \text{if $\mu\geq 0$}, \\
0, & \text{otherwise}.
\end{cases} \end{gather*}
For $\ddRH \in \{\de +\de f, \de f\}, \lambda \in \QQ$, let
\begin{gather*} F_X^\lambda(\ddRH) = F^\lambda(\ddRH) =
\big[ \sO_X(\flr{-\lambda P})_+ \xrightarrow{\ddRH} \cdots \xrightarrow{\ddRH} \Omega_X^i(\log S)(\flr{(i-\lambda)P})_+ \xrightarrow{\ddRH} \cdots \big]. \end{gather*}
The subindex $X$ in $F_X^\lambda(\ddRH)$ will be omitted if the base variety is clear. The \textit{Hodge filtrations} of the de Rham and the Higgs cohomologies are \begin{gather}\label{eq:H-fil}
F^\lambda\coH^k(U,\ddRH) = \mathrm{Image}\big\{ \HH^k(X,F^\lambda(\ddRH)) \to \coH^k(U,\ddRH) \big\}
\end{gather}
induced by the inclusions of complexes.

\item[(iv)] For $\alpha\in\QQ$, the \textit{Kontsevich sheaf of differential $p$-forms} is the coherent subsheaf of $\Omega_X^p(\log S)(*P)$
\begin{gather*} \Omega_f^p(\alpha) = \ker\big\{ \ddRH\colon \Omega_X^p(\log S)(\flr{\alpha P}) \to \Omega_X^{p+1}(*S)/\Omega_X^{p+1}(\log S)(\flr{\alpha P}) \big\}, \end{gather*}
and it forms the \textit{Kontsevich complex} $(\Omega_f^\bullet(\alpha), \ddRH)$ equipped with the filtration $(\Omega_f^\bullet(\alpha), \ddRH)_{\bullet\geq p}$ by direct truncation. We simply write $\Omega_f^p$ for $\Omega_f^p(0)$.
\end{enumerate}
\end{definition*}

\begin{proposition}\label{Prop:Kont-Log}\quad
\begin{enumerate}\itemsep=0pt
\item For $\alpha\in\QQ$ and $p\in\ZZ$, the $\sO_X$-module $\Omega_f^p(\alpha)$ is locally free of rank $\binom{\dim X}{p}$ with
\begin{gather*} \Omega_f^p(\alpha) = \Omega_f^p \otimes_{\sO_X}\sO_X(\flr{\alpha P}). \end{gather*}
\item For $\ddRH\in\{\de +\de f, \de f\}$ the three inclusions
\begin{gather*}
F^0(\ddRH) \to F^{-\alpha}(\ddRH), \qquad 0\leq \alpha, \\
(\Omega_f^\bullet(\alpha),\ddRH) \to F^{-\alpha}(\ddRH), \qquad 0\leq \alpha, \\
(\Omega_f^\bullet(\alpha),\ddRH)_{\bullet\geq p} \to F^{-\alpha +p}(\ddRH), \qquad 0\leq \alpha <1
\end{gather*}
are quasi-isomorphisms.
\end{enumerate}
\end{proposition}

\begin{proof} Both statements are local properties for coherent sheaves on $X$ and we can restrict to the coordinates such that \eqref{eq:non-deg-local} holds.

(i) In case $f = \frac{1}{x^e}$, the local freeness and the quasi-isomorphisms are obtained in \cite[equation~(1.3.1)]{ESY}, \cite[Lemma 2.12(a)]{KKP} and \cite[Proposition~1.3.]{Y}, \cite[Proposition~1.4.2]{ESY},
respectively. In fact, the methods are similar to the arguments below.

Consider the second case $f = \frac{z_1}{x^e}$ so that
\begin{gather*} z_1\frac{\de f}{f} = x^e\de f = \de z_1 - z_1\sum e_i\frac{\de x_i}{x_i}. \end{gather*}
In this chart,
the $\sO_X$-module $\Omega_X^p(\log S)$ is freely generated by
\begin{gather}\label{eq:with-df}
z_1\frac{\de f}{f}\wedge\bigwedge_{i=1}^{p-1}\xi_i, \qquad \{\xi_i \}_{i=1}^{p-1} \subset \left\{ \de z_2,\dots, \de z_r,\frac{\de x_1}{x_1},\dots, \frac{\de x_l}{x_l},\frac{\de y_1}{y_1},\dots, \frac{\de y_m}{y_m} \right\}
\end{gather}
and
\begin{gather}\label{eq:without-df}
\bigwedge_{i=1}^p\eta_i, \qquad \{\eta_i \}_{i=1}^p \subset \left\{ \de z_2,\dots, \de z_r, \frac{\de x_1}{x_1},\dots, \frac{\de x_l}{x_l},\frac{\de y_1}{y_1},\dots, \frac{\de y_m}{y_m} \right\}.
\end{gather}
The $\sO_X$-module $\Omega_f^p(\alpha)$ is indeed freely generated by
\begin{gather}\label{eq:Kont-loc-gen}
x^{-\flr{\alpha e}}\omega_1 \qquad\text{and}\qquad x^{e-\flr{\alpha e}}\omega_2,
\end{gather}
where $\omega_1$ and $\omega_2$ run through elements in \eqref{eq:with-df} and \eqref{eq:without-df}, respectively.

(ii) Let $D$ be a divisor supported on $\red{P}$, and $E$ an irreducible component of $\red{P}$. Let
\begin{gather*} A^p_k(D) = \Omega^p(\log S)(D +pP + kE), \qquad p,k\in\ZZ, \alpha\in\QQ. \end{gather*}
To obtain the quasi-isomorphisms, it suffices to show that the inclusion of complexes
\begin{gather}\label{eq:extended}
(A_{k-1}^\bullet(D), \ddRH) \to (A_k^\bullet(D), \ddRH)
\end{gather}
is a quasi-isomorphism for any $k,\alpha$ and all choices of $D$ (cf.~\cite[Proposition~1.2]{Y}). In fact, the first desired quasi-isomorphism then follows immediately. The last two quasi-isomorphisms can be derived by a decreasing induction (see the analogous statement \cite[Proposition~1.4.2]{ESY} in the case of good {\cpt} and its proof, which works for both $\ddRH = \de+\de f, \de f$).

To check the quasi-isomorphism \eqref{eq:extended}, we may assume that $E = (x_1)$ in the local model. Consider local sections $h\in\sO_X(D)$, and $\omega_1$ and $\omega_2$ in \eqref{eq:with-df} and \eqref{eq:without-df},
respectively. Then $\big\{h\cdot x^{-pe}x_1^{-k}\omega_i\big\}$ generates $A_k^p(D)$ and
\begin{align*}
\ddRH\colon \ & A_k^p(D)/A_{k-1}^p(D) \to A_k^{p+1}(D)/A_{k-1}^{p+1}(D), \\
& \frac{h}{x^{pe}x_1^k}\cdot \begin{cases}
\omega_1, \\ \omega_2
\end{cases}
 \mapsto \begin{cases}
0, \\
\frac{h}{x^{(p+1)e} x_1^k} (x^e\de f)\wedge\omega_2.
\end{cases}
\end{align*}
Therefore the quotient $\big(A_k^p(D)/A_{k-1}^p(D),\ddRH\big)_{p\in\ZZ}$ of \eqref{eq:extended} is exact.
\end{proof}

\begin{remark*} In a more functorial way, one can consider the $\sD$-module $\sM$ on $X$ attached to the meromorphic connection $(\sO_X(*S),\de +\de f)$ and define the \textit{irregular Hodge filtration} on $\sM$ as given in \cite[Definition~5.1]{SY}. By \cite[Lemma~9.17]{SY}, the filtered complex $F^\lambda(\de +\de f)$ defined above is quasi-isomorphic to the filtered de Rham complex associated with the filtered $\sD$-module $\sM$. When $X$ is a good {\cpt}, this is obtained in \cite[Proposition~1.7.4]{ESY}. Forgetting the filtrations, the quasi-isomorphisms of various de Rham complexes of $\sM$ are also derived in \cite[Lemmas 2.13 and 2.15]{M_GKZ} in the case where $X$ is non-degenerate, including the compactly supported counterpart.
\end{remark*}

\subsection{The independence}
\begin{proposition*}
For $\dRH \in \{\dR,\Hig\}$, the space $\coH_\dRH^k(U,f)$ with the Hodge filtration is independent of the choice of the non-degenerate {\cpt} of $(U,f)$. More precisely, if $\pi\colon X' \to X$ is a~morphism between non-degenerate {\cpt}s extending the identity on $U$, then there is a natural quasi-isomorphism
\begin{gather}\label{eq:RpiF}
\RD\pi_* F^\lambda_{X'}(\ddRH) \simeq F^\lambda_X(\ddRH), \qquad \lambda\in\QQ.
\end{gather}
\end{proposition*}

\begin{proof} The assertion for good {\cpt}s is proved in \cite[Theorem~1.7]{Y}, by comparing the degree $p$ components $F_X^\lambda(\ddRH)^p$ of $F_X^\lambda(\ddRH)$ on various $X$ for a fixed $p$ (and hence the proof works for both $\ddRH = \de +\de f$ and $\de f$).

In the following we show that by performing successively certain blowups, one can replace a~non-degenerate {\cpt} $X$ (thus in the second case of \eqref{eq:non-deg-local}) by a good one and compare the involved chain complexes (cf.~\cite[Section~4($b$)]{Y}). Let $\varpi \colon \widetilde{X} \to X$ be the blowup along the intersection $\Xi$ of $Z = (z_1)$ and the irreducible component $(x_1)$ of $P$ with multiplicity $e_1$. Let $E$ be the exceptional divisor, $\widetilde{S} = \widetilde{X}\setminus U$ and $\widetilde{P}$ the pole divisor of $\tilde{f} := \varpi^*f$. We want to establish that $\RD\varpi_*F_{\widetilde{X}}^\lambda(\ddRH)$ and $F_X^\lambda(\ddRH)$ are canonically quasi-isomorphic.

In case $e_1 >1$, it can be proved similarly as \cite[Proposition~4.4]{Y}. In more details, write $\lambda = -\alpha+p$ where $0\leq\alpha <1$ and $p\in\ZZ$. On $\widetilde{X}$ define the complex
\begin{gather*} R^\lambda(\ddRH) = \big(\Omega_{\widetilde{X}}^\bullet\big(\log \widetilde{S}\big)(\flr{(\alpha -p+\bullet )\big(\widetilde{P}+E)}\big)_+, \ddRH \big). \end{gather*}
We have
\begin{gather*} \varpi^*F_X^\lambda(\ddRH), F_{\widetilde{X}}^\lambda(\ddRH) \subset R^\lambda(\ddRH). \end{gather*}
By \cite[Proposition~4.4(i)]{Y}, each component of the complex $R^\lambda(\ddRH)/\varpi^*F_X^\lambda(\ddRH)$ is a direct sum of copies of $\sO_{E/\Xi}(-1)$. Hence the adjunction $F_X^\lambda(\ddRH) \to \RD\varpi_*R^\lambda(\ddRH)$ is a quasi-isomorphism. On the other hand, there are increasing complexes $R^\lambda_q(\ddRH)$ on $\widetilde{X}$ (those denoted by $R_\lambda(q)$ in \cite[equation~(27)]{Y}, which is a complex under either $\de +\de f$ or $\de f$) with $R^\lambda_{-1}(\ddRH) = F^\lambda_{\widetilde{X}}(\ddRH)$ and $R^\lambda_{\dim \widetilde{X} -p}(\ddRH) = R^\lambda(\ddRH)$ such that $R^\lambda_q(\ddRH)/R^\lambda_{q-1}(\ddRH)$
is quasi-isomorphic to the complex consisting of a direct sum of copies of $\sO_{E/\Xi}(-1)$ concentrated at degree $(p+q)$. In fact, in \cite[Lemma~4.5]{Y}, we further introduce the complexes $(K_{\rho,\eta,\xi}^\bullet, \de +\de f)$ and prove that the inclusion $(K_{\rho,\eta,\xi}^\bullet, \de +\de f) \subset (K_{\rho,\eta+1,\xi}^\bullet, \de +\de f)$ is a quasi-isomorphism by showing the quotient complex is exact. Now
one notices that $K_{\rho,\eta,\xi}^\bullet$ is indeed also stable under $\de f$; on the quotient $K_{\rho,\eta+1,\xi}^\bullet/K_{\rho,\eta,\xi}^\bullet$, one has the equality $\de +\de f = \de f$ of the differential maps
under the condition $e_1>1$. Therefore the proof of \cite[Proposition~4.4(ii)]{Y} describing $R^\lambda_q(\ddRH)/R^\lambda_{q-1}(\ddRH)$ goes through in both cases $\ddRH = \de +\de f$ and $\ddRH = \de f$. Hence one concludes that the inclusion $F_{\widetilde{X}}^\lambda(\ddRH) \subset R^\lambda(\ddRH)$ induces a quasi-isomorphism $\RD\varpi_*F_{\widetilde{X}}^\lambda(\ddRH) \to \RD\varpi_*R^\lambda(\ddRH)$ under the assumption $e_1>1$.

The following arguments work for any $e_1\geq 1$ and simplify those in \cite[Section~4($b$)]{Y}. (Cf., \cite[Lemma 2.12(b)]{KKP}.) By Proposition~\ref{Prop:Kont-Log}, it suffices to show that for any $0\leq\alpha <1$, $p\in\ZZ$, we have the inclusion relation $\varpi^*\Omega_f^p(\alpha) \subset \Omega_{\tilde{f}}^p(\alpha)$ on $\widetilde{X}$ and it induces an isomorphism $\Omega_f^p(\alpha) \to \RD\varpi_*\Omega_{\tilde{f}}^p(\alpha)$ on~$X$.

We first compute $\Omega_{\tilde{f}}^p(\alpha)$. Explicitly the blowup $\widetilde{X}$ is defined by the equation
\begin{gather*} \det \begin{pmatrix}
x_1 & z_1 \\
u & v
\end{pmatrix}
= 0 \end{gather*}
in $X\times\PP^1$ where $[u:v]$ is the homogeneous coordinate on $\PP^1$. On the chart $v\neq 0$ with local coordinates
\begin{gather*} \left\{ \bar{u} = \frac{u}{v},x_2,\dots, x_l, y_1, \dots, y_m,z_1, \dots, z_r \right\}, \end{gather*}
the $\sO_{\widetilde{X}}$-module $\Omega_{\tilde{f}}^p(\alpha)$ is generated by the basis
\begin{gather}\label{eq:chart-u-1}
\delta\frac{\de \tilde{f}}{\tilde{f}}\wedge\bigwedge_{i=1}^{p-1}\xi_i,
\qquad\begin{array}{@{}l}
\delta = z_1^{-\flr{\alpha(e_1-1)}}\bar{u}^{-\flr{\alpha e_1}}	x_2^{-\flr{\alpha e_2}}\cdots x_l^{-\flr{\alpha e_l}}, \\
\displaystyle \{\xi_i \}_{i=1}^{p-1} \subset \left\{ \frac{\de z_1}{z_1} + \frac{\de\bar{u}}{\bar{u}},\frac{\de x_2}{x_2},\dots, \frac{\de x_l}{x_l},\frac{\de y_1}{y_1},\dots, \frac{\de y_m}{y_m}, \de z_2,\dots, \de z_r \right\},
\end{array}
\end{gather}
and
\begin{gather}\label{eq:chart-u-2}
\delta\frac{1}{\tilde{f}}\bigwedge_{i=1}^p\eta_i,\qquad
\{\eta_i \}_{i=1}^p \subset\left\{ \frac{\de z_1}{z_1} + \frac{\de\bar{u}}{\bar{u}},\frac{\de x_2}{x_2},\dots, \frac{\de x_l}{x_l},\frac{\de y_1}{y_1},\dots, \frac{\de y_m}{y_m},\de z_2,\dots, \de z_r \right\}.
\end{gather}
On the chart $u\neq 0$ with local coordinates
\begin{gather*} \left\{
x_1,\dots, x_l, y_1, \dots, y_m, \bar{v} = \frac{v}{u},z_2, \dots, z_r \right\}, \end{gather*}
the sheaf $\Omega_{\tilde{f}}^p(\alpha)$ is generated by the basis
\begin{gather}\label{eq:chart-v-1}
\varepsilon\bar{v}\frac{\de \tilde{f}}{\tilde{f}}\wedge\bigwedge_{i=1}^{p-1}\xi_i,
\qquad\begin{array}{@{}l}
\varepsilon = x_1^{-\flr{\alpha(e_1-1)}}	x_2^{-\flr{\alpha e_2}}\cdots x_l^{-\flr{\alpha e_l}}, \\
\displaystyle\{\xi_i \}_{i=1}^{p-1} \subset \left\{ \frac{\de x_1}{x_1},\dots, \frac{\de x_l}{x_l},\frac{\de y_1}{y_1},\dots, \frac{\de y_m}{y_m},\de z_2,\dots, \de z_r \right\},
\end{array}
\end{gather}
and
\begin{gather}\label{eq:chart-v-2}
\varepsilon\frac{\bar{v}}{\tilde{f}}\bigwedge_{i=1}^p\eta_i, \qquad \{\eta_i \}_{i=1}^p \subset\left\{ \frac{\de x_1}{x_1},\dots, \frac{\de x_l}{x_l},\frac{\de y_1}{y_1},\dots, \frac{\de y_m}{y_m},\de z_2,\dots, \de z_r \right\}.
\end{gather}
On the intersection $u,v\neq 0$, one has $\frac{\de x_1}{x_1} = \frac{\de z_1}{z_1}+\frac{\de\bar{u}}{\bar{u}}$.

Using the basis \eqref{eq:Kont-loc-gen} of $\Omega_f^p(\alpha)$, a direct computation reveals that $\varpi^*\Omega_f^p(\alpha)$ is contained in~$\Omega_{\tilde{f}}^p(\alpha)$. Moreover, the $\sO_E$-module $\Omega_{\tilde{f}}^p(\alpha)/\varpi^*\Omega_f^p(\alpha)$ equals either zero if $\flr{\alpha(e_1-1)} \neq \flr{\alpha e_1}$ or otherwise $\binom{l+m+r}{p}$ copies of $\sO_{E/\Xi}(-1)$ generated by \eqref{eq:chart-u-1}, \eqref{eq:chart-u-2} and \eqref{eq:chart-v-1}, \eqref{eq:chart-v-2} on the two charts, respectively. Therefore we obtain that $F_X^\lambda(\ddRH)$ and $\RD\varpi_*F_{\widetilde{X}}^\lambda(\ddRH)$ are naturally quasi-isomorphic.

One then iteratively takes the blowups along the intersections of irreducible components of the zero and the pole divisors as in \cite[Section~4$(b)$]{Y} (the diagram (26) therein) to obtain a~good {\cpt} $X'$ of $(U, f)$ from the non-degenerate $X$ with the canonical quasi-isomor\-phism~\eqref{eq:RpiF}.
\end{proof}

\begin{remark*}\quad\begin{enumerate}\itemsep=0pt
\item[(i)] By the $E_1$-degeneration \cite[Theorem~1.2.2]{ESY}, \cite[Theorem~1.1]{M_Kont} on a good {\cpt} (see also \cite[Theorem 2.18]{KKP}), the above proposition implies that the arrow in \eqref{eq:H-fil} is injective for any non-degenerate $X$ and indices $k,\lambda$.

\item[(ii)] For $X$ non-degenerate, we have $\coH_\dR^k(U,f) = \HH^k(X,(\Omega_X^\bullet(*S),\de + \de f))$ by the arguments of \cite[Corollary~1.4]{Y}. On the other hand, $\coH_\Hig^k(U,f) \neq \HH^k(X,(\Omega_X^\bullet(*S),\de f))$ in general which can be seen by considering the case $U$ affine and $f=0$.
\end{enumerate}
\end{remark*}

\section{The proof of the main result}\label{sect:Resolution}

We begin with two pairs $(U_i,f_i)$, $i=1,2$, and their product $(U,f)$. Fix good {\cpt}s~$X_i$ of $(U_i,f_i)$ such that $S_i := X_i\setminus U_i$ are strict normal crossing divisors. The proof of Theorem~\ref{Thm:Kunneth} consists of two steps. In the first step Section~\ref{sect:elim}, we construct explicitly a non-degenerate {\cpt} $X$ of~$(U,f)$ from $X_1\times X_2$ by successive blowups. In step two Section~\ref{sect:relations}, we compare the filtered de Rham or the Higgs complex on~$X$, which gives the Hodge filtration on~$\coH_\dRH^k(U,f)$, with a certain filtered complex on~$X_1\times X_2$ that gives the product filtration using the explicit construction of~$X$.

\subsection{An explicit elimination}\label{sect:elim}

For each $i=1,2$, take an open covering of $X_i$ with a system of local coordinates
\begin{gather*} \{ x_{i,1},\dots, x_{i,l_i}, 	x_{i,l_i+1},\dots,x_{i,l_i+m_i}, 	x_{i,l_i+m_i+1},\dots,x_{i,l_i+m_i+r_i} \} \end{gather*}
such that
\begin{gather*} f_i = \frac{1}{x_{i,1}^{e_{i,1}}\cdots x_{i,l_i}^{e_{i,l_i}}}, \quad e_{i,j} > 0, \qquad S_i = (x_{i,1}\cdots x_{i,l_i+m_i}). \end{gather*}
As for the initial data in the inductive construction, we consider the {\cpt} $X_1\times X_2$ of $U$ with the systems of local coordinates $\{x_{i,j}\}$.

Suppose we have constructed a {\cpt} $Y$ of $U$ and its systems of local coordinates
\begin{gather}\label{eq:Y_LocCoor}
\{ y_1,\dots, y_l, y_{l+1}, \dots, y_{l+m},	y_{l+m+1}, \dots, y_{l+m+r} \},
\end{gather}
together with a birational map $\pi\colon Y \to X_1\times X_2$ such that
\begin{gather}\label{eq:Y_LocBeh}
\pi^*f_1 = \frac{1}{y_1^{a_1}\cdots y_l^{a_l}}, \qquad \pi^*f_2 = \frac{1}{y_1^{b_1}\cdots y_l^{b_l}}, \qquad Y\setminus U = (y_1\cdots y_{l+m})
\end{gather}
for some $a_i,b_i \geq 0$ with $b_i >0$ if $a_i = 0$. Let $T$ be the boundary divisor $Y\setminus U$. For a pair of irreducible components $D_1$, $D_2$ of $T$, set
\begin{gather*} \Delta_Y(D_1,D_2) = \begin{cases}
\displaystyle \prod_{i=1}^2\big(\ord_{D_i}(\pi^*f_1)-\ord_{D_i}(\pi^*f_2)\big), 	& \text{if $D_1\neq D_2, D_1\cap D_2 \neq \varnothing$}, \\
0, & \text{otherwise}. \end{cases}\end{gather*}

Notice that if $\Delta_Y(D_1,D_2) \geq 0$ for any pair $(D_1, D_2)$, i.e., $a\geq b$ or $b\geq a$ in terms of the systems of local coordinates as above, then the zero divisor of $\pi^*f$ is smooth in a neighborhood of $T$
and intersects $T$ transversally. That is, $Y$ is a non-degenerate {\cpt} of $(U,f)$ in this case.

Otherwise, pick a pair $(D_1,D_2)$ such that $\Delta_Y(D_1,D_2) <0$ and is the smallest among all possible values of $\Delta_Y$. Let $\widetilde{Y}$ be the blowup of $Y$ along $D_1\cap D_2$ and $\tilde{\pi}\colon \widetilde{Y} \to X_1\times X_2$ the induced map. Then $\widetilde{T}:= \widetilde{Y}\setminus U$ consists of the exceptional divisor $E$ and $\{\widetilde{D}\}$ where $\widetilde{D}$ denotes the proper transform of an irreducible component $D$ of $T$. To construct the explicit systems of local coordinates of $\widetilde{Y}$, we may assume that $D_i = (y_i)$ for $i=1,2$ with $a_1 > b_1$ and $a_2 < b_2$ after rearrangement. The blowup is defined by the equation $y_1v=y_2u$ where $[u:v]$ is the homogeneous coordinate of $\PP^1$. Over this chart of coordinates of $Y$, we add two charts to $\widetilde{Y}$ (and away from the blowup center, we pass the charts of $Y$ to $\widetilde{Y}$). In the chart $v\neq 0$ of~$\PP^1$, we consider the local coordinates
\begin{gather}\label{eq:Yv_LocCoor}
\left\{ \bar{u}:= \frac{u}{v}, y_2, y_3, \dots, y_{l+m+r} \right\}.
\end{gather}
Then
\begin{gather}\label{eq:Yv_LocBeh}
\tilde{\pi}^*f_1 = \frac{1}{\bar{u}^{a_1}y_2^{a_1+a_2}	y_3^{a_3}\cdots y_l^{a_l}},\qquad \tilde{\pi}^*f_2 = \frac{1}{\bar{u}^{b_1}y_2^{b_1+b_2}	y_3^{b_3}\cdots y_l^{b_l}},
\\\widetilde{T} = (\bar{u}y_2\cdots y_{l+m}),\qquad \widetilde{D}_1 = (\bar{u}), \qquad E = (y_2).\nonumber \end{gather}
In the chart $u\neq 0$ of $\PP^1$, we consider the local coordinates
\begin{gather}\label{eq:Yu_LocCoor}
\left\{ \bar{v}:= \frac{v}{u}, y_1, y_3, \dots, y_{l+m+r} \right\}.
\end{gather}
Then
\begin{gather}\label{eq:Yu_LocBeh}
\tilde{\pi}^*f_1 = \frac{1}{\bar{v}^{a_2}y_1^{a_1+a_2}	y_3^{a_3}\cdots y_l^{a_l}},\qquad \tilde{\pi}^*f_2 = \frac{1}{\bar{v}^{b_2}y_1^{b_1+b_2}	y_3^{b_3}\cdots y_l^{b_l}},
\\ \widetilde{T} = (\bar{v}y_1y_3\cdots y_{l+m}),\qquad \widetilde{D}_2 = (\bar{v}), \qquad E = (y_1). \nonumber\end{gather}
One has $\widetilde{D}_1\cap\widetilde{D}_2 = \varnothing$ and
\begin{gather*}
\Delta_{\widetilde{Y}}(\widetilde{D},\widetilde{D}') = \Delta_Y(D,D') \qquad \text{for $(D,D') \neq (D_1,D_2), (D_2,D_1)$}, \\
\Delta_{\widetilde{Y}}(\widetilde{D}_i,E) = \Delta_Y(D_1,D_2) + \big(\ord_{D_i}(\pi^*f_1)-\ord_{D_i}(\pi^*f_2)\big)^2	> \Delta_Y(D_1,D_2), \\
\Delta_{\widetilde{Y}}(\widetilde{D},E)= \Delta_Y(D,D_1) + \Delta_Y(D,D_2).
\end{gather*}
Notice that
\begin{gather*} \Delta_Y(D,D_i) = \left(\ord_D(\pi^*f_1) -\ord_D(\pi^*f_2)\right)\cdot (-a_i+b_i) \end{gather*}
and consequently
\begin{gather*} \Delta_{\widetilde{Y}}(\widetilde{D},E) \ \begin{cases}
> \min\{\Delta_Y(D,D_1),\Delta_Y(D,D_2)\} & \text{if $\ord_D(\pi^*f_1) \neq \ord_D(\pi^*f_2)$}, \\
= 0 & \text{if $\ord_D(\pi^*f_1) = \ord_D(\pi^*f_2)$}.
\end{cases} \end{gather*}
We then replace $Y$ with its systems of local coordinates by $\widetilde{Y}$ with the coordinates constructed above.

Observe that after a finite number of blowups in this procedure, the smallest possible value of the function $\Delta$, if it is negative on $Y$, strictly increases. Hence repeating this construction, it produces a~non-degenerate {\cpt} $X$ of $(U,f)$ obtained by a sequence $X \to\cdots\to X_1\times X_2$ of explicit blowups.

\subsection{Relations between complexes}\label{sect:relations}
We shall put a filtered complex on each step of the sequence of blowups constructed in Section~\ref{sect:elim} and compare them under push-forwards.

Consider a birational map $\pi\colon Y \to X_1\times X_2$ and let $P_{Y,f_j}$ be the pole divisor of the pullback of $f_j$ on $Y$. For any $\lambda\in\QQ$, consider the subsheaf $\sO_Y^{(\lambda)}$ of $\sO_Y(*(P_{Y,f_1}+P_{Y,f_2}))$ given by
\begin{gather*} \sO_Y^{(\lambda)} = \sum_{0\leq\theta\leq1} \sO_Y\left(\flr{\lambda((1-\theta) P_{Y,f_1} + \theta P_{Y,f_2})}\right), \end{gather*}
which is coherent but not locally free in general. Let $T = Y\setminus \pi^{-1}(U)$ and
\begin{gather*} \Omega_Y^{p,(\lambda)} = \sO_Y^{(\lambda)} \otimes_{\sO_Y} \Omega_Y^p(\log T). \end{gather*}
For $\ddRH \in \{\de +\de f, \de f\}$, consider the complex on $Y$
\begin{gather*} F_Y^{(\lambda)}(\Theta) = \big[ \sO_Y^{(-\lambda)_+} \xrightarrow{\ddRH} \Omega_Y^{1,(1-\lambda)_+}	\xrightarrow{\ddRH} \Omega_Y^{2,(2-\lambda)_+}	\xrightarrow{\ddRH} \cdots \big], \end{gather*}
where
\begin{gather*} \Omega_Y^{p,(p-\lambda)_+} = \begin{cases}
\Omega_Y^{p,(p-\lambda)}, & \text{if $p-\lambda \geq 0$}, \\
0 ,& \text{otherwise}. \end{cases} \end{gather*}
(It is indeed a sub-complex of $(\Omega_Y^\bullet(*T),\ddRH)$.)

If $Y$ is a non-degenerate {\cpt} of $(U,f)$, e.g., $Y$ equals the iterated blowup $X$ constructed in the previous subsection, then $F_Y^{(\lambda)}(\ddRH)$ is indeed the filtration defining the desired Hodge filtration.
The following two lemmas describe the situations in the initial step $Y = X_1\times X_2$ and in each blowup $\varpi\colon \widetilde{Y}\to Y$ appeared in the sequence occurred in Section~\ref{sect:elim}, respectively.

\begin{lemma*} Consider $Y = X_1\times X_2$ and let $F_\boxtimes^\lambda(\ddRH)$ be the product filtration of $F_{X_i}^\lambda(\ddRH_i)$ whose $p$-th component is
\begin{gather*} F_\boxtimes^\lambda(\ddRH)^p := \sum_{a\in\QQ, q\in\mathbb{Z}} 	F_{X_1}^a(\ddRH_1)^q \boxtimes F_{X_2}^{\lambda-a}(\ddRH_2)^{p-q} \end{gather*}
inside $\Omega_Y^p(*T)$. $($Here $\ddRH_i = \de+\de f_i$ and $\de f_i$ if $\ddRH = \de+\de f$ and $\de f$, respectively.$)$
\begin{enumerate}\itemsep=0pt
\item The filtration $F_Y^{(\lambda)}(\ddRH)$ coincides with $F_\boxtimes^\lambda(\ddRH)$.
\item For all $\lambda \geq \rho$, the induced map
\begin{gather*} \HH^k\big(Y,F_\boxtimes^\lambda(\ddRH)\big) \to \HH^k\big(Y,F_\boxtimes^\rho(\ddRH)\big) \end{gather*}
is injective; it is an isomorphism if $\lambda\leq 0$. One has
\begin{gather}\label{eq:Kunneth-Y1}
\HH^k\big(Y, F_\boxtimes^\lambda(\ddRH)\big) = \bigoplus_{i+j=k}\left(\sum_{a +b = \lambda} \HH^i\big(X_1,F_{X_1}^a(\ddRH_1)\big)\otimes\HH^j\big(X_2,F^b_{X_2}(\ddRH_2)\big) \right)
\end{gather}
for any $\lambda$ where the inner sum is taken inside the vector space $\coH^i(U_1,\ddRH_1)\otimes \coH^j(U_2,\ddRH_2)$.
\end{enumerate}
\end{lemma*}

\begin{proof}(i) Indeed the inclusion $F_\boxtimes^\lambda(\ddRH) \subset F_Y^{(\lambda)}(\ddRH)$ is clear. Conversely suppose that $\mu := p-\lambda$ $\geq 0$. The $p$-th degree component $\Omega_Y^{p,(\mu)}$ of $F_Y^{(\lambda)}(\ddRH)$ is generated by products of elements in
\begin{gather*} \Omega_{X_1}^q(\log S_1)(\flr{(1-\theta)\mu P_{X_1}})\qquad\text{and}\qquad \Omega_{X_2}^{p-q}(\log S_2)(\flr{\theta\mu P_{X_2}}), \end{gather*}
which are the $q$-th and the $(p-q)$-th components of $F_{X_1}^{q-(1-\theta)\mu}(\ddRH_1)$ and $F_{X_2}^{p-q-\theta\mu}(\ddRH_2)$, respectively. The latter two contribute to $F_\boxtimes^\lambda(\ddRH)$.

(ii) We have the natural external products
\begin{gather*} F_{X_1}^a(\ddRH_1)\boxtimes F_{X_2}^b(\ddRH_2)	\to F_\boxtimes^\lambda(\ddRH) \end{gather*}
for all $a+b = \lambda$. Using \v{C}ech resolution or representatives in smooth forms, one obtains the cup product
\begin{gather*} \HH^i\big(X_1,F_{X_1}^a(\ddRH_1)\big)\otimes	\HH^j\big(X_2,F^b_{X_2}(\ddRH_2)\big)\to \HH^{i+j}\big(Y,	F_{X_1}^a(\ddRH_1)\boxtimes F_{X_2}^b(\ddRH_2)\big). \end{gather*}
On the other hand, one has from the definition that
\begin{gather*} \Gr_{F_\boxtimes(\ddRH)}^\lambda = \bigoplus_{a+b=\lambda} \Gr_{F_{X_1}(\ddRH_1)}^a \boxtimes \Gr_{F_{X_2}(\ddRH_2)}^b. \end{gather*}
Again there is the cup product
\begin{gather}\label{eq:cup_gr}
\bigoplus_{i+j=k} \HH^i\big(X_1,\Gr_{F_{X_1}(\ddRH_1)}^a\big)\otimes \HH^j\big(X_2,\Gr_{F_{X_2}(\ddRH_2)}^b\big)\to \HH^k\big(Y, \Gr_{F_{X_1}(\ddRH_1)}^a \boxtimes \Gr_{F_{X_2}(\ddRH_2)}^b \big).
\end{gather}
To complete the assertions, it suffices to show that the arrow above is an isomorphism. Indeed, denote the sum inside the big round brackets in the right side of \eqref{eq:Kunneth-Y1} by $\Phi(i,j,\lambda)$. The $E_1$-degeneration of the spectral sequence attached to each filtered complex $F_{X_i}(\ddRH_i)$ implies that there is the natural exact sequence
\begin{gather*} 0 \to \bigcup_{\mu>\lambda} \Phi(i,j,\mu) \to \Phi(i,j,\lambda) \to\bigoplus_{a+b=\lambda} \HH^i\big(X_1,\Gr_{F_{X_1}(\ddRH_1)}^a\big) \otimes \HH^j\big(X_2,\Gr_{F_{X_2}(\ddRH_2)}^b\big) \to 0. \end{gather*}
Together with the isomorphism of \eqref{eq:cup_gr}, a decreasing induction on the index $\lambda$ in \eqref{eq:Kunneth-Y1} then gives the desired statements. Finally the isomorphism of \eqref{eq:cup_gr} can be obtained by directly truncating the involved complexes and inductively using the K\"unneth formula for coherent sheaves \cite[Theorem~1]{SW}.
\end{proof}

\begin{lemma*} Let $\varpi\colon \widetilde{Y} \to Y$ be the blowup constructed in Section~{\rm \ref{sect:elim}}. Then the pullback $\varpi^*F_Y^{(\lambda)}(\ddRH)$ is a sub-complex of $F_{\widetilde{Y}}^{(\lambda)}(\ddRH)$
and the following hold.
\begin{enumerate}\itemsep=0pt
\item Each component $F_{\widetilde{Y}}^{(\lambda)}(\ddRH)^p/\varpi^*F_Y^{(\lambda)}(\ddRH)^p$ of the quotient is supported on the exceptional divisor $E$ and is a direct sum of copies of relative $\sO(-1)$ of the $\PP^1$-bundle $E$ over the center of the blowup $\varpi$.
\item We have the canonical quasi-isomorphism $F_Y^{(\lambda)}(\ddRH) \to \RD\varpi_*F_{\widetilde{Y}}^{(\lambda)}(\ddRH)$.
\end{enumerate}
\end{lemma*}

\begin{proof}(i) Let $\widetilde{T} = \widetilde{Y}\setminus U$. As in the proof of the previous lemma, assume that $\mu:= p-\lambda \geq 0$. Since $\varpi^*\Omega_Y^p(\log T) = \Omega_{\widetilde{Y}}^p(\log \widetilde{T})$,
we only need to consider the difference between $\sO_{\widetilde{Y}}^{(\mu)}$ and $\varpi^*\sO_Y^{(\mu)}$. We use the local coordinates~\eqref{eq:Y_LocCoor} with the properties~\eqref{eq:Y_LocBeh}. Then $\sO_Y^{(\mu)}$ is generated by various
\begin{gather}\label{eq:g_Gen}g := \frac{1}{y^{\flr{c}}},\end{gather}
where $c = (1-\theta)\mu a + \theta\mu b$ for $0\leq\theta\leq 1$. On the other hand, $\sO_{\widetilde{Y}}^{(\mu)}$ is generated by various
\begin{gather*} h_{\bar{u}} := \frac{1}{\bar{u}^{\flr{c_1}}y_2^{\flr{c_1+c_2}}y_3^{\flr{c_3}}\cdots y_l^{\flr{c_l}}} \qquad\text{and}\qquad h_{\bar{v}} :=	\frac{1}{\bar{v}^{\flr{c_2}}y_1^{\flr{c_1+c_2}}y_3^{\flr{c_3}}
	\cdots y_l^{\flr{c_l}}} \end{gather*}
on the two charts \eqref{eq:Yv_LocCoor} and \eqref{eq:Yu_LocCoor} satisfying \eqref{eq:Yv_LocBeh} and \eqref{eq:Yu_LocBeh}, respectively. We have
\begin{gather*} \gen{h_{\bar{u}}}_{\sO_{\widetilde{Y}}}/\varpi^*\gen{g}_{\sO_Y} = \begin{cases}
\gen{h_{\bar{u}}}_{\sO_E}, & 	\text{if $\flr{c_1+c_2} = \flr{c_1} + \flr{c_2} + 1$}, \\
0, & \text{if $\flr{c_1+c_2} = \flr{c_1} + \flr{c_2}$},
\end{cases} \end{gather*}
and similarly on the chart $u\neq 0$. Observe on the intersection that $h_{\bar{u}} = \bar{v}^{-1}h_{\bar{v}}$ if $\flr{c_1+c_2} \neq \flr{c_1} + \flr{c_2}$. Thus $\sO_{\widetilde{Y}}^{(\mu)}/\varpi^*\sO_{Y}^{(\mu)}$
is a direct sum of copies of the relative $\sO(-1)$.

(ii) There are only finitely many $y^{-\flr{c}}$ occurred in the generators in~\eqref{eq:g_Gen}; call the appeared (distinct) monomials $\xi_1,\dots,\xi_k$ ordered by the increment of the corresponding parameter $\theta$.
Consider locally the filtration $M_i = \gen{\xi_1,\dots,\xi_i}$ of $\sO_Y$-submodules of $\sO_Y^{(\mu)}$. On this chart one has the short exact sequence
\begin{align*}
0 \to N_{i+1} \to M_i\oplus \sO_Y\cdot\xi_{i+1} & \to M_{i+1} \to 0, \\
(\omega, \eta) & \mapsto \omega - \eta.
\end{align*}
Here if $\xi_i = \prod y_j^{-c_{i,j}}$, then $c_{i,j}$ is monotone as a function of $i$ and $N_{i+1} = M_i\cap \sO_Y\xi_{i+1}$ is the invertible sheaf generated by
\begin{gather*} \prod y_j^{-\min\{c_{i,j}, c_{i+1,j} \}} \end{gather*}
(with diagonal embedding to $M_i\oplus \sO_Y\xi_{i+1}$). Applying inductively the projection formula (e.g., \cite[Exercise~III.8.3]{H}) to the locally free $\sO_Y\xi_{i+1}$ and $N_i$, one obtains that
\begin{gather*} \RD\varpi_*\varpi^*M_i= \RD\varpi_*\big(\sO_{\widetilde{Y}}\otimes\varpi^*M_i\big)= \RD\varpi_*\sO_{\widetilde{Y}}\otimes M_i = M_i, \end{gather*}
since $\RD\varpi_*\sO_{\widetilde{Y}} = \sO_Y$ for the birational morphism $\varpi$. Together with the computation in~(i), the assertion follows.
\end{proof}

The proof of the main Theorem~\ref{Thm:Kunneth} is now completed by the above two lemmas.

\section{The Brieskorn lattice}\label{sect:BL}

We indicate that the K\"unneth formula for $\coH_\dR^k(U,f)$ and $\coH_\Hig^k(U,f)$ can be put together into a~family version (cf.~\cite[Sections~1.3 ans~6.2]{SY}, \cite[Section~3.2]{KKP} and \cite{M_Kont}).

Consider the affine line $\AA^1_u$ with a fixed coordinate $u$.

Fix a pair $(U,f)$ and a non-degenerate {\cpt} $X$. Let $\pi\colon X\times \AA^1_u \to \AA^1_u$ be the projection. The \textit{$k$-th Brieskorn lattice of $(U,f)$} (cf.~\cite[Section~6.1]{SY}) is the coherent sheaf on $\AA^1_u$
\begin{gather*} \sG^k(U,f) := \RD^k\pi_*\big(\Omega_f^q(\alpha)\boxtimes\sO_{\AA^1_u},	u\cdot\de_X +\de f \big)_{q\geq 0}. \end{gather*}
Here $\de_X$ is the derivative with respective to the component $X$ only.

Let
\begin{gather*} F^{-\alpha+p}\sG^k(U,f) = \RD^k\pi_* \big(\Omega_f^q(\alpha)\boxtimes\sO_{\AA^1_u}, u\cdot\de_X +\de f \big)_{q\geq p}, \qquad 0\leq \alpha <1. \end{gather*}
According to the quasi-isomorphisms in Proposition~\ref{Prop:Kont-Log}(ii) and the $E_1$-degeneration
\begin{gather*} \dim \coH^k(U,\ddRH) = 	\sum_q \dim \coH^{k-q}\big(X,\Omega_f^q(\alpha)\big), \qquad \ddRH \in \{\de +\de f, \de f \}, \end{gather*}
we have that
\begin{itemize}\itemsep=0pt
\item the $\sO_{\AA^1_u}$-module $\sG^k(U,f)$ is free and independent of the choice of $\alpha$, and
\item the canonical maps $F^{-\alpha+p}\sG^k(U,f) \to \sG^k(U,f)$ indeed define a filtration by free subsheaves of $\sG^k(U,f)$ with free quotients whose fibers are
\begin{gather}\label{eq:Kont-fiber}
F^{-\alpha+p}\sG^k(U,f)\rest{u=c} = \begin{cases}
F^{-\alpha+p}\coH_\Hig^k(U,f), & c=0, \\
F^{-\alpha+p}\coH_\dR^k(U,f/c), & c\neq 0
\end{cases}
\end{gather}
under the base-change map (cf.\ the arguments in the proof of \cite[Proposition~1.5.1]{ESY}).
\end{itemize}

Now consider as in Section~\ref{sect:Resolution} two pairs $(U_i, f_i)$ and their product $(U,f)$. We again have the natural map
\begin{gather*} \bigoplus_{0\leq i\leq k} \sG^i(U_1,f_1) \otimes_{\sO_{\AA^1_u}} \sG^{k-i}(U_2,f_2) \to \sG^k(U,f) \end{gather*}
obtained by cup product. Similar to the proof of \cite[Proposition~1.5.1]{ESY}, the fiber-wise K\"unneth formula Theorem~\ref{Thm:Kunneth} shows that the above map is an isomorphism strictly compatible with the filtrations.

On the other hand, fix $0\leq \alpha <1$. One can naturally complete $\big(\sG^k(U,f), F^{-\alpha +p}\big)_{p \in\ZZ}$ into a~filtered bundle on $\PP^1$ by adding the filtered cohomology space
\begin{gather*} F^p\coH_{\de,\alpha}^k(U,f) :=\mathrm{Image}\big\{\HH^k\big(X, (\Omega_f^\bullet(\alpha),\de)_{\bullet \geq p} \big)\to \HH^k\big(X, (\Omega_f^\bullet(\alpha),\de)\big) \big\} \end{gather*}
as the fiber over $u = \infty$. The resulting sheaf on $\PP^1$ is called the \textit{Kontsevich bundle} and denoted by $\sK_\alpha^k(U,f)$. The filtered space $\coH_{\de,\alpha}^k(U,f)$ depends on $\alpha$ but does not depend on the choice of the non-degenerate {\cpt} $X$ since in fact (\cite[Theorem~1.11(a), Section~1.3]{SY}, \cite[Theorem~1.2(i,ii)]{M_Kont})
\begin{enumerate} \itemsep=0pt
\item the bundle $\sK_\alpha^k(U,f)$ can be obtained (as a Deligne extension) by using a natural algebraic connection on $\sG^k(U,f)$ (see \cite[Lemma 6.2]{SY} with the aid of Proposition~\ref{Prop:Kont-Log}(ii) in the case of non-degenerate {\cpt} $X$) with regular singularity at $u=\infty$, and
\item under the base-change, one has
\begin{gather}\label{eq:d_HN}
\big(\coH_{\de,\alpha}^k(U,f), F^\bullet\big) = \big(\sK_\alpha^k(U,f), \HN^\bullet\big)\rest{u=\infty},
\end{gather}
where $\HN^p\sK^k_\alpha(U,f)$ is the Harder--Narasimhan filtration on the locally free sheaf $\sK^k_\alpha(U,f)$ normalized with $\Gr_{\HN}^p$ isomorphic to a direct sum of copies of $\sO(p)$ on $\PP^1$.
\end{enumerate}
In fact, we also have
\begin{gather}\label{eq:ddRH_HN}
F^{-\alpha+\bullet}\sG^k(U,f) =\big(\sK_\alpha^k(U,f), \HN^\bullet\big)\rest{u\neq\infty}.
\end{gather}

However, one does not have the direct analogue of the K\"unneth formula for $\coH_{\de,\alpha}^\bullet$ in general. For example, in the case $(U_i,f_i) = (\AA^1,x_i^2)$ where $x_i$ is a global coordinate on the affine line~$\AA^1$, one has
\begin{gather*} \dim \coH_{\de,0}^k\big(\AA^1,x_i^2\big) = \begin{cases}
1, & \text{if $k = 1$}, \\
0, & \text{otherwise},
\end{cases} \end{gather*}
and the cup product
\begin{gather}\label{eq:cup_zero}
\coH_{\de,0}^1\big(\AA^1,x_1^2\big) \otimes \coH_{\de,0}^1\big(\AA^1,x_2^2\big) \to \coH_{\de,0}^2\big(\AA^2,x_1^2+x_2^2\big)
\end{gather}
is zero. The vanishing can be obtained by a careful cohomology calculation or manipulating Kontsevich bundles as follows.

Indeed, consider the complex $\big[\Omega_{f_i}^0 \xrightarrow{\de} \Omega_{f_i}^1\big]$ on $\PP^1$, whose hypercohomology gives $\coH_{\de,0}^k\big(\AA^1,x_i^2\big)$. Let $\pi\colon \widetilde{X} \to \PP^1\times\PP^1$ be the blowup at $(\infty,\infty)$. Then $\widetilde{X}$ is a non-degenerate {\cpt} of $\big(\AA^2, f=f_1+f_2\big)$. On $\widetilde{X}$ we have the inclusion
\begin{gather*} A:= \pi^*\big(\boxtimes_{i=1}^2 (\Omega_{f_i}^\bullet,\de) \big) \to B:= (\Omega_f^\bullet, \de) \end{gather*}
and the arrow \eqref{eq:cup_zero} factors through the induced map $\HH^2\big(\widetilde{X}, A\big) \to \HH^2\big(\widetilde{X}, B\big)$. One checks that the last map is zero.

On the other hand, for $\ddRH = \de + \de x_i^2$ or $\de x_i^2$, we have
\begin{gather*} \coH^1\big(\AA^1, \ddRH\big) = \Gr_F^{1/2} \coH^1\big(\AA^1, \ddRH\big) = \coH^0\big(\PP^1, \Omega^1(2[\infty])\big), \end{gather*}
which is generated by the class $\de x_i$. By \eqref{eq:Kont-fiber}, \eqref{eq:d_HN}, \eqref{eq:ddRH_HN}, one obtains that
\begin{gather*} \sK^1_\alpha\big(\AA^1,x_i^2\big) = \begin{cases}
\sO_{\PP^1}\cdot\de x_i, & 0\leq \alpha < \frac{1}{2}, \\
\sO_{\PP^1}([\infty])\cdot\de x_i, & \frac{1}{2}\leq \alpha <1,
\end{cases} \end{gather*}
and in particular, $\coH_{\de,0}^1\big(\AA^1,x_i^2\big)$ is generated by $\de x_i$ as a fiber of $\sK^1_0\big(\AA^1,x_i^2\big)$. Similar argument shows that
\begin{gather*} \sK^2_\alpha\big(\AA^2,x_1^2+x_2^2\big) = \sO_{\PP^1}([\infty])\cdot\de x_1\de x_2, \qquad 0\leq \alpha <1. \end{gather*}
We conclude that the map \eqref{eq:cup_zero} sends the generator $\de x_1\otimes \de x_2$ to $\de x_1\de x_2$, which represents zero in $\coH_{\de,0}^2\big(\AA^2,x_1^2+x_2^2\big)$ as a fiber of $\sK^2_0\big(\AA^2,x_1^2+x_2^2\big)$.

\subsection*{Acknowledgements}

The second author thanks Noriko Yui for providing him the wonderful opportunity to work with her during the years 2006 to 2008 and the happy lunch time at Queen's University. We thank the referee for the careful reading,
helpful comments and in particular pointing out the insufficiency of the proof of Proposition~\ref{Prop:Kont-Log} in an earlier version. This work was partially supported by the MoST and the NCTS, Taiwan.

\pdfbookmark[1]{References}{ref}
\LastPageEnding


\begin{thebibliography}{10}
\footnotesize\itemsep=0pt

\bibitem{ESY}
Esnault H., Sabbah C., Yu J.-D., {$E_1$}-degeneration of the irregular {H}odge
 filtration, \href{https://doi.org/10.1515/crelle-2014-0118}{\textit{J.~Reine Angew. Math.}} \textbf{729} (2017), 171--227,
 \href{https://arxiv.org/abs/1302.4537}{arXiv:1302.4537}.

\bibitem{H}
Hartshorne R., Algebraic geometry, \href{https://doi.org/10.1007/978-1-4757-3849-0}{\textit{Graduate Texts in Mathematics}}, Vol.~52, Springer-Verlag, New York~--
 Heidelberg, 1977.

\bibitem{KKP}
Katzarkov L., Kontsevich M., Pantev T., Bogomolov--{T}ian--{T}odorov theorems
 for {L}andau--{G}inzburg models, \href{https://doi.org/10.4310/jdg/1483655860}{\textit{J.~Differential Geom.}} \textbf{105}
 (2017), 55--117, \href{https://arxiv.org/abs/1409.5996}{arXiv:1409.5996}.

\bibitem{M_Kont}
Mochizuki T., A twistor approach to the {K}ontsevich complexes,
 \href{https://doi.org/10.1007/s00229-017-0989-5}{\textit{Manuscripta Math.}}, {t}o appear, \href{https://arxiv.org/abs/1501.04145}{arXiv:1501.04145}.

\bibitem{M_GKZ}
Mochizuki T., Twistor property of {GKZ}-hypergeometric systems,
 \href{https://arxiv.org/abs/1501.04146}{arXiv:1501.04146}.

\bibitem{M_twistor}
Mochizuki T., Mixed twistor {$\mathcal D$}-modules, \href{https://doi.org/10.1007/978-3-319-10088-3}{\textit{Lecture Notes in
 Mathematics}}, Vol.~2125, Springer, Cham, 2015.

\bibitem{S}
Sabbah C., Monodromy at infinity and {F}ourier transform, \href{https://doi.org/10.2977/prims/1195145150}{\textit{Publ. Res.
 Inst. Math. Sci.}} \textbf{33} (1997), 643--685.

\bibitem{S2}
Sabbah C., Irregular {H}odge theory, \href{https://arxiv.org/abs/1511.00176}{arXiv:1511.00176}.

\bibitem{SY}
Sabbah C., Yu J.-D., On the irregular {H}odge filtration of exponentially
 twisted mixed {H}odge modules, \href{https://doi.org/10.1017/fms.2015.8}{\textit{Forum Math. Sigma}} \textbf{3} (2015),
 e9, 71~pages, \href{https://arxiv.org/abs/1406.1339}{arXiv:1406.1339}.

\bibitem{SW}
Sampson J.H., Washnitzer G., A {K}\"unneth formula for coherent algebraic
 sheaves, \textit{Illinois~J. Math.} \textbf{3} (1959), 389--402.

\bibitem{Y}
Yu J.-D., Irregular {H}odge filtration on twisted de {R}ham cohomology,
 \href{https://doi.org/10.1007/s00229-013-0642-x}{\textit{Manuscripta Math.}} \textbf{144} (2014), 99--133, \href{https://arxiv.org/abs/1203.2338}{arXiv:1203.2338}.

\end{thebibliography}
\end{document}